\documentclass[12pt,a4paper,reqno]{amsart}
\usepackage{amsmath,amssymb,amsthm}
\usepackage[T1]{fontenc}
\usepackage[utf8]{inputenc}
\usepackage[english]{babel}
\usepackage{enumerate}

\advance\textwidth30mm \advance\hoffset-14mm
\advance\textheight30mm \advance\voffset-18mm
\newtheorem{thm}{Theorem}
\newtheorem{lem}{Lemma}
\newtheorem{cor}{Corollary}

\theoremstyle{remark}
\newtheorem{rem}{Remark}
\theoremstyle{definition}

\begin{document}

\title[Parametric Families of Minimal Positive Cubature Formulas]{Parametric Families of Minimal Positive Cubature Formulas of Degrees 4
and 6 on the 2-Sphere}
\author{D.\,V.~Gorbachev}
\address{Saint Petersburg State University}
\email{dvgmail@mail.ru}
\date{}

\begin{abstract}
We study minimal positive cubature formulas on the two-dimensional sphere of
degrees four and six, with ten and eighteen nodes, respectively. We prove the
regularity of Popov's symmetric formulas and establish that their local families
of deformations are two-dimensional, up to rotations. We obtain algebraic
relations between the weights and pairwise inner products of the nodes. For
degree four, we construct a one-parameter family with half-turn symmetry
connecting the formulas of Popov and Heo--Xu. We also give a self-contained
proof that eighteen nodes are minimal for degree six.
\end{abstract}



\maketitle

\section{Introduction}

\subsection{Positive cubature formulas}
Let $\mathbb S^2\subset\mathbb R^3$ be the unit sphere, $\sigma$ the
normalized surface measure on it, and $\Pi_t(\mathbb S^2)$ the space of
spherical polynomials of degree at most $t\in\mathbb Z_{\ge0}$, with
$D_{t}=\dim\Pi_t(\mathbb S^2)=(t+1)^2$.

A positive cubature formula of degree $t$, or equivalently a positive
weighted spherical $t$-design, is a pair $(X,\Lambda)$, where
$X=\{x_1,\dots,x_N\}\subset\mathbb S^2$ is a set of distinct nodes and
$\Lambda=(\lambda_1,\dots,\lambda_N)\in\mathbb R_{>0}^N$ is a vector of weights
satisfying
\begin{equation}\label{eq:def}
\sum_{j=1}^N\lambda_jp(x_j)= \int_{\mathbb S^2}p\,d\sigma, \quad
\forall\,p\in\Pi_t(\mathbb S^2).
\end{equation}
We call a cubature formula minimal if its number of nodes is the smallest
among all positive formulas of degree $t$. We denote this minimum by
$N_{\mathrm{wt}}(t)$. Formulas obtained from one another by rotating the sphere
are naturally regarded as equivalent.

The notion of a spherical design was introduced by Delsarte, Goethals, and
Seidel~\cite{DelsarteGoethalsSeidel} by analogy with classical combinatorial
designs. Their Fisher-type lower bound for the number of nodes remains valid
for arbitrary positive weights: on $\mathbb S^2$ it takes the form
$N\ge(m+1)^2$ for $t=2m$ and $N\ge(m+1)(m+2)$ for $t=2m+1$.

In~\cite{BondarenkoGorbachev}, the existence of a two-parameter family of
minimal positive cubature formulas of degree $4$ with $10$ nodes was noted.
The authors stated that the pairwise inner products $u=x_i\cdot x_j$,
$i\ne j$, and the corresponding weights satisfy the relation
\[
225u^{4}+180u^{3}-54u^{2}-36u-315
+36(\lambda_i^{-1}+\lambda_j^{-1})
-4(\lambda_i\lambda_j)^{-1}=0.
\]
This relation was found by extensive computer calculations, including the
determination of minimal polynomials of algebraic numbers. No theoretical
derivation was given in~\cite{BondarenkoGorbachev}. In the present paper we
fill this gap.

We have $N_{\mathrm{wt}}(4)=10$. The equality $N_{\mathrm{wt}}(6)=18$ was
recently announced by Cheng and Yu~\cite{ChengYu}. For completeness, we give
a short independent proof of this fact below.

To state our main results, we need two symmetric formulas of
Popov~\cite{Popov}, which will serve as the starting points of our
deformations. For an orthogonal transformation $Q\in O(3)$ of finite order,
we denote the orbit of $x\in\mathbb S^2$ by
\[
\mathcal O_Q(x)=\{Q^k x\colon k=0,\ldots,\operatorname{ord}(Q)-1\}.
\]
For $\lambda>0$, we write $\mathcal O_Q(x;\lambda)$ for the weighted orbit
$\mathcal O_Q(x)$ in which each node has weight $\lambda$.
Consider cylindrical coordinates on the sphere:
\[
x(h,\varphi)= \bigl(\sqrt{1-h^2}\cos\varphi, \sqrt{1-h^2}\sin\varphi,h\bigr),
\]
where $h\in[-1,1]$ is the height and $\varphi\in[0,2\pi)$ the azimuth.
Let $G\in O(3)$ be the orthogonal transformation of order $8$ defined by
\[
Gx(h,\varphi)= x\Bigl(-h,\varphi+\frac{\pi}{4}\Bigr).
\]
For $0<|h|<1$, the orbit $\mathcal O_G(x(h,\varphi))$ consists of the vertices
of a square antiprism at heights $\pm h$; when $h=0$, it is a regular octagon
on the equator. Let $G_0(\xi_1,\xi_2,\xi_3)=(\xi_1,\xi_2,-\xi_3)$ denote
reflection in the equatorial plane. Then $\mathcal O_{G_0}(e_3)=\{\pm e_3\}$.

For $t=4$, Popov's formula is
\[
(X_4^{\mathrm P},\Lambda_4^{\mathrm P})= \mathcal
O_{G_0}\Bigl(e_3;\frac1{12}\Bigr) \cup \mathcal
O_G\Bigl(x\Bigl(\frac1{\sqrt5},0\Bigr);\frac5{48}\Bigr).
\]
For $t=6$, Popov's formula is
\[
(X_6^{\mathrm P},\Lambda_6^{\mathrm P})= \mathcal
O_{G_0}\Bigl(e_3;\frac2{45}\Bigr) \cup \mathcal
O_G\bigl(x(h_+,0);\lambda_+\bigr) \cup \mathcal
O_G\Bigl(x\Bigl(h_-,\frac{\pi}{4}\Bigr);\lambda_-\Bigr),
\]
where
\begin{equation}\label{h-lambda}
h_{\pm}=\sqrt{\frac{13\pm2\sqrt{30}}{49}}, \quad
\lambda_\pm=\frac{410\pm\sqrt{30}}{7200}.
\end{equation}

At the poles, the transformations $G$ and $G_0$ agree, while on each
antiprism $G$ cyclically permutes the nodes without changing their weights.
Thus both Popov formulas are invariant under $G$.

Another minimal $10$-node weighted $4$-design was constructed by Heo and
Xu~\cite{HeoXu}. The configuration has a common polar axis and consists of one pole, a square,
and a pentagon, with the latter two at different heights relative to that
axis. The pole and one vertex of the pentagon have weight $1/9$.

\subsection{Main result}
Fix a sequence of real spherical harmonics $Y_1,Y_2,\ldots,$ that is
orthonormal in $L^2(\mathbb S^2,\sigma)$ and ordered by degree so that, for
every $t\ge0$, the harmonics $Y_1,\ldots,Y_{D_t}$ form an orthonormal basis of
$\Pi_t(\mathbb S^2)$. We write
\[
K_{t}(x\cdot y)=\sum_{j=1}^{D_{t}}Y_{j}(x)Y_{j}(y),\quad K_{t}(1)=D_{t},
\]
for the reproducing kernel of $\Pi_t(\mathbb S^2)$.

Consider the moment map $\Phi_t\colon (\mathbb S^2)^N\times\mathbb
R_{>0}^N\to\mathbb R^{D_t}$:
\[
\Phi_t(X,\Lambda)= \biggl(\sum_{j=1}^N\lambda_jY_k(x_j) -\int_{\mathbb
S^2}Y_k\,d\sigma\biggr)_{k=1}^{D_t}.
\]
Then condition~\eqref{eq:def} is equivalent to the equation
$\Phi_t(X,\Lambda)=0$. If $(X_0,\Lambda_0)$ is a solution of this equation
and the differential $\mathrm d\Phi_t(X_0,\Lambda_0)$ has full rank $D_t$,
then, by the implicit function theorem, the solution set is locally a
real-analytic manifold of dimension ${3N-D_t}$.

\begin{thm}\label{thm:main-families}
For $t=4$ and $6$, Popov's formula
$(X_t^{\mathrm P},\Lambda_t^{\mathrm P})$ is a regular point of the map~$\Phi_t$:
\[
\operatorname{rank}\mathrm d\Phi_t(X_t^{\mathrm P},\Lambda_t^{\mathrm P})=D_t.
\]
Up to equivalence under rotations, the local dimension of the set of positive
cubature formulas of degree $t$ near Popov's formula is $2$.
\end{thm}

This theorem establishes the existence of local two-dimensional families,
but we do not obtain a complete explicit description of these families.
For $t=4$, however, we can identify within such a family an explicit
one-parameter branch with the additional symmetry of the half-turn
$R\in SO(3)$,
\[
R(\xi_1,\xi_2,\xi_3)=(-\xi_1,-\xi_2,\xi_3),
\]
connecting the formulas of Popov and Heo--Xu.

\begin{cor}\label{cor:one-parameter-4}
There exists a one-parameter family of minimal positive cubature formulas of
degree $4$ with $10$ nodes, invariant under the half-turn $R$, with parameter
\[
0\le a\le a'=\frac{(4-\sqrt6)^{3/2}}{125}.
\]
Each formula in this family has the form
\[
(X(a),\Lambda(a))=
\bigcup_{i=1}^5\mathcal O_R\bigl(x(h_i,\varphi_i);\lambda_i\bigr),\quad
\lambda_i=\frac{1}{30h_i^4-24h_i^2+12},
\]
where the heights $h_1,\ldots,h_5$ are the roots, counted with multiplicity,
of the polynomial
\[
\omega_a(u)=u^5+\frac{25}{2}\,au^4-\frac45\,u^3-8au^2+\frac{2}{25}\,u+a.
\]
The azimuths $\varphi_i$ can be chosen so that
\[
\sum_{i=1}^5\lambda_i(1-h_i^2)h_i^k e^{2i\varphi_i}=0,\quad k=0,1,2,\quad
\sum_{i=1}^5\lambda_i(1-h_i^2)^2e^{4i\varphi_i}=0.
\]
For $a=0$, we obtain a formula equivalent to Popov's formula, while for
$a=a'$ we obtain a formula equivalent to that of Heo--Xu.
\end{cor}

\subsection{Organization of the paper}
In Section~\ref{sec:auxiliary}, we introduce the matrices associated with a
cubature formula, derive rank relations, and prove a criterion for regularity
of the moment map. Section~\ref{sec:regularity} deals with the regularity of
Popov's formulas and local parameters for their deformations. In
Section~\ref{sec:one-parameter}, we construct a one-parameter family of
degree-four formulas connecting the formulas of Popov and Heo--Xu.
Section~\ref{app:min6} gives a proof of the minimal number of nodes in
degree six.
\section{Design matrices and a regularity criterion for the moment map}
\label{sec:auxiliary}

We retain the notation $\Phi_t$, $Y_j$, $D_t$, and $K_t$ introduced above.
Here $I_N$ denotes the identity matrix of size $N$, and $C\succeq0$
means that the matrix is positive semidefinite.

\begin{lem}\label{lem:projectors}
Let $(X,\Lambda)$ be a positive cubature formula of degree $t=2s$.
Consider the evaluation matrix
\[
V_{ij}=\sqrt{\lambda_i}\,Y_j(x_i),\quad
i=1,\ldots,N,\quad j=1,\ldots,D_s.
\]
Then $V^{T}V=I_{D_s}$, and the symmetric matrices
\[
A=VV^{T},\quad B=I_N-A,\quad
C_{ij}=\frac{B_{ij}}{\sqrt{\lambda_i\lambda_j}}
\]
satisfy the following properties:

\begin{enumerate}[\rm (i)]
\item $A$ and $B$ are orthogonal projectors;
\item $\operatorname{rank}A=D_s$, $\operatorname{rank}B=
\operatorname{rank}C=N-D_s$;
\item the following identities hold:
\begin{equation}\label{eq:ABC}
A_{ij}=\sqrt{\lambda_i\lambda_j}K_s(x_i\cdot x_j),\quad
A_{ii}=\lambda_iD_s;
\end{equation}
\item we have $C\succeq0$ and
\[
C_{ii}=\lambda_i^{-1}-D_s,\quad C_{ij}=-K_s(x_i\cdot x_j),\quad i\ne j.
\]
\end{enumerate}

In particular, $N\ge D_s$ and $\lambda_i\le D_s^{-1}$.
\end{lem}

\begin{proof}
Since $Y_iY_j\in\Pi_{2s}(\mathbb S^2)$, it follows from~\eqref{eq:def} that
\[
(V^{T}V)_{ij}=\sum_{k=1}^N\lambda_kY_i(x_k)Y_j(x_k)
=\int_{\mathbb S^2}Y_iY_j\,d\sigma=\delta_{ij}.
\]
Thus $\operatorname{rank}V=D_s$ and
$A^2=V(V^{T}V)V^{T}=A$. Identities~\eqref{eq:ABC} follow from the definition
of $K_s$ and the equality $K_s(1)=D_s$. The matrix $B$ is the complementary
orthogonal projector. The matrix $C$ is obtained from $B$ by multiplication
on the left and right by the invertible diagonal matrix
$\operatorname{diag}(\lambda_1^{-1/2},\ldots,\lambda_N^{-1/2})$,
so its rank and positive semidefiniteness are preserved.
Finally, $B_{ii}=1-\lambda_iD_s\ge0$.
\end{proof}

\begin{cor}
For a $10$-node formula of degree $4$, the matrix $C$ has rank $1$, and
\begin{equation}\label{eq:4-basic}
K_2(x_i\cdot x_j)^2=(\lambda_i^{-1}-9)(\lambda_j^{-1}-9),\quad i\ne j.
\end{equation}
For an $18$-node formula of degree $6$, the matrix $C$ has rank $2$.
Consequently, for any distinct $i$, $j$, $k$,
\begin{equation}\label{eq:6-rank-relation}
\det\begin{pmatrix}
\lambda_i^{-1}-16 & -K_3(x_i\cdot x_j) & -K_3(x_i\cdot x_k)\\
-K_3(x_i\cdot x_j) & \lambda_j^{-1}-16 & -K_3(x_j\cdot x_k)\\
-K_3(x_i\cdot x_k) & -K_3(x_j\cdot x_k) & \lambda_k^{-1}-16
\end{pmatrix}=0.
\end{equation}
\end{cor}

\begin{proof}
Apply Lemma~\ref{lem:projectors}: $10-D_2=1$ and $18-D_3=2$.
The relations follow from the vanishing of the corresponding principal minors.
\end{proof}

In the cases under consideration,
\[
K_2(u)=\frac{15u^2+6u-3}{2},\quad K_3(u)=\frac{35u^3+15u^2-15u-3}{2}.
\]
Thus~\eqref{eq:4-basic}, with $u=x_i\cdot x_j$, $i\ne j$, becomes
\[
225u^{4}+180u^{3}-54u^{2}-36u-315
+36(\lambda_i^{-1}+\lambda_j^{-1})-4(\lambda_i\lambda_j)^{-1}=0.
\]
This is precisely the relation stated in~\cite{BondarenkoGorbachev}. For degree
$6$, its natural analogue is the three-point
relation~\eqref{eq:6-rank-relation}. Both relations are necessary for a
configuration to be a positive cubature formula of the corresponding degree,
but neither is by itself equivalent to the moment condition
$\Phi_t(X,\Lambda)=0$. A rank-based approach to ruling out spherical
$4$-designs with equal weights was also used in~\cite{G10,SSP}.

To study deformations, we need the following regularity criterion for the
moment map. We write $\nabla_Tq$ for the spherical gradient of
$q\in\Pi_t(\mathbb S^2)$.

\begin{lem}\label{lem:rankcriterion}
Suppose all weights are positive. Then
$\operatorname{rank}\mathrm d\Phi_t(X,\Lambda)=D_t$ if and only if, for any
$q\in\Pi_t(\mathbb S^2)$, the conditions
\begin{equation}\label{q-q}
q(x_j)=0,\quad \nabla_Tq(x_j)=0,\quad j=1,\ldots,N,
\end{equation}
imply $q=0$.
\end{lem}

\begin{proof}
Let $J$ be the matrix of the differential $\mathrm d\Phi_t(X,\Lambda)$
in local coordinates. Then
$\operatorname{rank}\mathrm d\Phi_t(X,\Lambda)=D_t$ if and only if
$\ker J^T=\{0\}$.

Put $b=(b_1,\ldots,b_{D_t})^T\in\ker J^T$ and
$q=\sum_{k=1}^{D_t}b_kY_k$. We have
\[
\sum_{k=1}^{D_t}b_k(\Phi_t)_k= \sum_{j=1}^N\lambda_jq(x_j)-\int_{\mathbb
S^2}q\,d\sigma.
\]
Differentiating with respect to $(X,\Lambda)$ gives
\[
\sum_{k=1}^{D_t}b_k\,\mathrm d(\Phi_t)_k
=
\sum_{j=1}^N
\bigl(q(x_j)\,\mathrm d\lambda_j
+\lambda_j\nabla_Tq(x_j)\cdot\mathrm dx_j\bigr).
\]
Since $J^Tb=0$, the left-hand side vanishes identically. The variations of
the weights and nodes are independent, and $\lambda_j>0$; hence
conditions~\eqref{q-q} hold. Conversely, if conditions~\eqref{q-q} hold for the corresponding $q$,
then $J^Tb=0$. This proves the lemma.
\end{proof}

\begin{rem}\label{rem-node}
If the positions of certain nodes are subject to additional constraints,
the condition $\nabla_Tq(x_j)=0$ is replaced by the requirement that
$\nabla_Tq(x_j)$ be orthogonal to all admissible directions of variation of
$x_j$. For a fixed node, this condition is absent.

When computing the rank, one may use any basis of
$\Pi_t(\mathbb S^2)$ instead of $Y_1,\ldots,Y_{D_t}$, including a
nonorthogonal basis: a change of basis multiplies the moment map and its
differential $\mathrm d\Phi_t$ on the left by a fixed invertible matrix,
and therefore does not change the rank of the differential.
\end{rem}

\section{Regularity of Popov's formulas}\label{sec:regularity}

We use the following nonorthogonal basis of $\Pi_t(\mathbb S^2)$ over
$\mathbb C$:
\[
f_{m,k}(h,\varphi)=(1-h^2)^{|m|/2}h^ke^{im\varphi},\quad -t\le m\le t,\quad
0\le k\le t-|m|.
\]
For $m\ge0$, these are the restrictions of the polynomials
$(\xi_1+i\xi_2)^m\xi_3^k$; for $m<0$, they are the restrictions of
$(\xi_1-i\xi_2)^{|m|}\xi_3^k$. Decomposition into azimuthal frequencies
shows that they are linearly independent, and their number is $D_t$.
We have
\[
\int_{\mathbb S^2}f_{m,k}\,d\sigma=
\begin{cases}
1/(k+1), & \text{$m=0$, $k$ even},\\
0, & \text{otherwise},
\end{cases}\quad
f_{m,k}(Gx)=\zeta^{m+4k}f_{m,k}(x),
\]
where $\zeta=e^{\pi i/4}$. Partition the basis into subsets
\[
\mathcal B_r=\{f_{m,k}\colon m+4k\equiv r\pmod8\},\quad r=0,\ldots,7,
\]
ordered first by increasing $m$ and then by increasing $k$. Then
\[
\Pi_t(\mathbb S^2)=\bigoplus_{r=0}^7E_r,\quad
E_r=\operatorname{span}\mathcal B_r,\quad
q(Gx)=\zeta^rq(x)\quad(q\in E_r).
\]
Passing to complex coefficients does not change the rank of the real
matrix of the moment map.

\begin{lem}\label{lem:popov-exactness}
For $t=4,6$, the formulas $(X_t^{\mathrm P},\Lambda_t^{\mathrm P})$ are
exact on $\Pi_t(\mathbb S^2)$.
\end{lem}

\begin{proof}
Both formulas and the measure $\sigma$ are invariant under $G$. Therefore, for
$r\ne0$, the cubature sum and the integral of every function in $E_r$ vanish,
so it suffices to verify exactness on $\mathcal B_0$. For $t=4$, this basis
consists of $1$, $h^2$, $h^4$, and
\[
2\,\frac1{12}+8\,\frac5{48}\Bigl(\frac1{\sqrt5}\Bigr)^{2j}
=\frac1{2j+1},\quad j=0,1,2.
\]
For $t=6$, we have
$\mathcal B_0=\{1,h^2,h^4,h^6,f_{-4,1},f_{4,1}\}$.
The required equalities
\[
2\,\frac2{45}+8\lambda_+h_+^{2j}+8\lambda_-h_-^{2j}
=\frac1{2j+1},\quad j=0,1,2,3,
\]
\[
\lambda_+h_+(1-h_+^2)^2=\lambda_-h_-(1-h_-^2)^2
\]
follow from~\eqref{h-lambda}.
\end{proof}
\subsection{Regularity for $t=4$}
Number the nodes as follows:
\[
x_1=e_3,\quad x_2=Gx_1=-e_3,\quad
x_3=x\Bigl(\frac1{\sqrt5},0\Bigr),\quad
x_{3+k}=G^kx_3,\quad k=0,\ldots,7.
\]

\begin{lem}\label{lem:rank4}
We have
\[
\operatorname{rank}\mathrm{d}\Phi_4(X_4^{\mathrm P},\Lambda_4^{\mathrm P})=25.
\]
The rank remains $25$ when both poles are fixed.
\end{lem}

\begin{proof}
The subspace of polynomials satisfying condition~\eqref{q-q} of
Lemma~\ref{lem:rankcriterion} is invariant under $G$, and hence also
under projection onto the subspaces $E_r$. Fix $r$ and take any $q\in E_r$.
We need to check that the conditions
\[
q(x_j)=0,\quad \nabla_Tq(x_j)=0,\quad j=1,\ldots,10,
\]
imply $q=0$. For $q\in E_r$ we have
\[
q(G^l x)=\zeta^{rl}q(x),\quad \nabla_Tq(G^l x)=\zeta^{rl}G^l\nabla_Tq(x).
\]
Thus the conditions on the two orbits reduce to those at their
representatives $x_1$ and $x_3$.

In cylindrical coordinates, the conditions at $x_3$ take the form
\[
q(x_3)=0,\quad \partial_hq(x_3)=0,\quad \partial_\varphi q(x_3)=0.
\]
Of the conditions at the pole, we use only $q(x_1)=0$. Write
\[
q=\sum_{f_{m,k}\in\mathcal B_r}c_{m,k}f_{m,k}.
\]
For convenience, multiply the condition $\partial_\varphi q(x_3)=0$
by $i^{-1}$. Then
\[
\begin{pmatrix}
q(x_3)\\
\partial_hq(x_3)\\
i^{-1}\partial_\varphi q(x_3)\\
q(x_1)
\end{pmatrix}=
\sum_{f_{m,k}\in\mathcal B_r}c_{m,k}
\begin{pmatrix}
f_{m,k}(x_3)\\ \sqrt5(k-|m|/4)f_{m,k}(x_3)\\ mf_{m,k}(x_3)\\ \delta_{m0}
\end{pmatrix}=0,
\]
where
\[
f_{m,k}(x_3)=\Bigl(\frac45\Bigr)^{|m|/2}5^{-k/2}.
\]
For $r=0$, the third row vanishes identically, while for
$r\ne0,4$, the fourth row vanishes. Let $M_r$ be the square matrix
formed by the remaining rows, with columns ordered according to
$\mathcal B_r$. For example,
\[
\mathcal B_0=(f_{0,0},f_{0,2},f_{0,4}),\quad
M_0=
\begin{pmatrix}
1&1/5&1/25\\
0&2/\sqrt5&4/(5\sqrt5)\\
1&1&1
\end{pmatrix},
\]
so that $\det M_0=32\sqrt5/125$. Similarly,
\[
\begin{array}{c|c|c}
r&\dim E_r&\det M_r\\ \hline
0&3&32\sqrt5/125\\
1&3&-256\sqrt5/625\\
2&3&-512/625\\
3&3&256/625\\
4&4&-12288/3125\\
5&3&-256/625\\
6&3&512/625\\
7&3&256\sqrt5/625
\end{array}
\]
All determinants are nonzero, so all coefficients $c_{m,k}$ vanish.
Since we did not use derivatives at the poles, Remark~\ref{rem-node}
shows that the rank remains unchanged when both poles are fixed.
\end{proof}

\subsection{Regularity for $t=6$}
Number the nodes as follows:
\[
x_1=e_3,\quad x_2=Gx_1=-e_3,\quad x_3=x(h_+,0),\quad
x_{11}=x\Bigl(h_-,\frac\pi4\Bigr),
\]
\[
x_{3+k}=G^kx_3,\quad x_{11+k}=G^kx_{11},\quad k=0,\ldots,7.
\]

\begin{lem}\label{lem:rank6}
We have
\[
\operatorname{rank}\mathrm d\Phi_6(X_6^{\mathrm P},\Lambda_6^{\mathrm P})=49.
\]
The rank remains $49$ when both poles are fixed.
\end{lem}

\begin{proof}
Repeating the argument in the proof of Lemma~\ref{lem:rank4},
for each $r$ it suffices to check conditions~\eqref{q-q} at
representatives $x_1$, $x_3$, and $x_{11}$ of the three orbits. Write
\[
q=\sum_{f_{m,k}\in\mathcal B_r}c_{m,k}f_{m,k}.
\]
At a point $x=x(h,\varphi)$, the conditions $q(x)=0$,
$\nabla_Tq(x)=0$ give
\[
\begin{pmatrix}
q(x)\\
\partial_hq(x)\\
i^{-1}\partial_\varphi q(x)
\end{pmatrix}=
\sum_{f_{m,k}\in\mathcal B_r}
c_{m,k}f_{m,k}(h,\varphi)
\begin{pmatrix}
1\\ k/h-|m|h/(1-h^2)\\ m
\end{pmatrix}.
\]
Apply this identity at $(h,\varphi)=(h_+,0)$ and
$(h_-,\pi/4)$ and add the condition
\[
q(x_1)=\sum_{f_{0,k}\in\mathcal B_r}c_{0,k}=0.
\]

For $r\ne0,4$, the pole row vanishes identically. For $r=4$, we use
all seven rows, placing the pole row last. For $r=0$, the two rows
corresponding to $\partial_\varphi q$ are proportional. We therefore
replace the row $i^{-1}\partial_\varphi q(x_3)$ by the pole row,
keeping its position in the matrix.

For convenience, multiply the three rows corresponding to $x_{11}$ by
$e^{-ir\pi/4}$ and denote the resulting square matrices by $M_r$.
This does not affect their nonsingularity. Moreover, for
$f_{m,k}\in\mathcal B_r$ we have $e^{i(m-r)\pi/4}=(-1)^k$, so all entries
of $M_r$ are real. Put $\gamma=\sqrt{30}$ and use the identities
\[
49h_+^4-26h_+^2+1=0,\quad h_-=\frac1{7h_+},\quad
\gamma=\frac{49h_+^2-13}{2},
\]
We obtain
\[
\begin{array}{c|c|c}
r&\dim E_r&\det M_r\\ \hline
0&6&-4246732800(3+41\gamma)/7^{15}\\
1,7&6&\mp679477248000\sqrt6(10-\gamma)h_+/7^{16}\\
2,6&6&\pm3261490790400\gamma/7^{16}\\
3,5&6&\pm135895449600\sqrt6(10-\gamma)h_+/7^{15}\\
4&7&-6522981580800\gamma/7^{16}
\end{array}
\]
All determinants are nonzero, so all coefficients $c_{m,k}$ vanish.
Since we did not use derivatives at the poles, Remark~\ref{rem-node}
shows that the rank remains unchanged when both poles are fixed.
\end{proof}

\subsection{Local parameters}
We keep the node numbering fixed under small deformations.

\begin{lem}\label{lem:local-chart}
For $t=4,\,6$, there exists $\varepsilon_t>0$ such that, whenever
$\tau_1^2+\tau_2^2<\varepsilon_t^2$, the conditions
\begin{equation}\label{eq:local-gauge}
x_1=e_3,\quad \varphi_3=0,\quad
x_2=\bigl(\tau_1,\tau_2,-\sqrt{1-\tau_1^2-\tau_2^2}\bigr)
\end{equation}
together with the equation $\Phi_t(X,\Lambda)=0$ uniquely determine,
near Popov's formula, all remaining nodes and all weights. These depend
real-analytically on $(\tau_1,\tau_2)$ and describe all sufficiently
nearby solutions subject to $x_1=e_3$, $\varphi_3=0$.
\end{lem}

\begin{proof}
With both poles fixed, the number of variables is
$3N-4=D_t+1$, and the rank is $D_t$ by
Lemmas~\ref{lem:rank4} and~\ref{lem:rank6}. Thus the kernel is
one-dimensional and coincides with the direction of rotations about
the $e_3$ axis. The condition $\varphi_3=0$ eliminates this direction,
so the Jacobian with respect to the remaining $3N-5=D_t$ variables
is nonsingular. The implicit function theorem with parameters
$(\tau_1,\tau_2)$ gives the result.
\end{proof}

\begin{proof}[Proof of Theorem~\ref{thm:main-families}]
Full rank has been established in Lemmas~\ref{lem:rank4} and~\ref{lem:rank6}.
The conditions $x_1=e_3$, $\varphi_3=0$ locally fix the rotational
freedom, so Lemma~\ref{lem:local-chart} gives a two-dimensional
local family up to equivalence under rotations.
\end{proof}

\begin{cor}
For $t=4,6$, any formula from Lemma~\ref{lem:local-chart} sufficiently
close to Popov's formula with $(\tau_1,\tau_2)\ne(0,0)$ contains no
antipodal pairs, and its weights are not all equal. For $t=6$,
its degree of exactness is exactly $6$.
\end{cor}

\begin{proof}
In Popov's formula, the poles form the only antipodal pair, while
all other quantities $|x_i+x_j|$ are positive. These properties,
as well as the inequality of some of the weights, persist under small
deformations. Moreover, by~\eqref{eq:local-gauge}, the equality
$x_1=-x_2$ is possible only when $(\tau_1,\tau_2)=(0,0)$.

For $t=6$, the polynomial
$p_7(\xi)=\xi_3^3\operatorname{Re}((\xi_1+i\xi_2)^4)$ has
zero integral, whereas its cubature sum in Popov's formula equals
\[
8\lambda_+h_+(1-h_+^2)^2(h_+^2-h_-^2) =\frac{16\sqrt3}{735}>0.
\]
By continuity, the sum remains nonzero in a sufficiently small
neighborhood.
\end{proof}
\section{A one-parameter family of degree four}\label{sec:one-parameter}

We prove Corollary~\ref{cor:one-parameter-4}. Consider formulas of the form
\begin{equation}\label{eq:4-family-nodes}
(X,\Lambda)=\bigcup_{j=1}^5\mathcal O_R\bigl(x(h_j,\varphi_j);\lambda_j\bigr),
\quad -1<h_j<1,
\end{equation}
where all ten nodes are distinct. The axis of the half-turn $R$ need not
coincide with the polar axis in the original representation of Popov's
formula. We first derive necessary relations for the heights and weights,
and then construct the formulas on the entire interval $0\le a\le a'$.

\begin{lem}
For a positive formula~\eqref{eq:4-family-nodes} of degree $4$, we have
\begin{equation}\label{eq:4-weight}
\lambda_j=\lambda(h_j),\quad
\lambda(h)=\frac1{30h^4-24h^2+12}.
\end{equation}
\end{lem}

\begin{proof}
By Lemma~\ref{lem:projectors}, $B=bb^T$, where $|b|=1$. The permutation of the
nodes induced by $R$ has even and odd subspaces, each of dimension $5$, in
$\mathbb R^{10}$. In $\Pi_2(\mathbb S^2)$, these dimensions are $5$ and $4$:
the even subspace is spanned by $1$, $\xi_3$, $\xi_3^2$, $\xi_1^2-\xi_2^2$, and
$\xi_1\xi_2$. The isometric embedding $V$ is equivariant under $R$, so its
one-dimensional orthogonal complement is odd. Thus the components of $b$ are
opposite within each pair. For a pair at height $h_j$, relation~\eqref{eq:ABC}
gives
\[
\lambda_jK_2(2h_j^2-1)=1-9\lambda_j.
\]
Substituting $K_2(2h^2-1)=30h^4-24h^2+3$, we obtain~\eqref{eq:4-weight}.
\end{proof}

\begin{lem}
If the heights of a positive formula~\eqref{eq:4-family-nodes} of
degree $4$ are distinct, then their monic polynomial has the form
\[
\omega_a(u)=u^5+\frac{25}{2}\,au^4-\frac45\,u^3-8au^2+\frac{2}{25}\,u+a
\]
for some $a\in\mathbb R$.
\end{lem}

\begin{proof}
For two different orbits, put
\[
y=h_i,\quad z=h_j,\quad d=\sqrt{(1-y^2)(1-z^2)}\cos(\varphi_i-\varphi_j).
\]
Their inner products are $yz\pm d$. The oddness of the defect vector
from the preceding proof gives
\[
K_2(yz+d)+K_2(yz-d)=0,\quad
K_2(yz+d)^2=K_2(2y^2-1)K_2(2z^2-1).
\]
The first equality means that $d^2=(1-2yz-5y^2z^2)/5$.
Eliminating $d$ from the second yields
\begin{multline*}
R_4(y,z)=625y^4z^4-400(y^4z^2+y^2z^4)+50(y^4+z^4) +100y^3z^3+320y^2z^2\\
{}-40(y^2+z^2)-8yz+4=0.
\end{multline*}
Let $\omega(u)=u^5+bu^4+\ldots$ be the polynomial of the heights.
For $y=h_i$, the other four heights are roots of $R_4(y,z)$, so
\[
R_4(y,z)=L(y)\,\frac{\omega(z)}{z-y},\quad L(y)=625y^4-400y^2+50.
\]
Here $L(y)\ne0$: otherwise $R_4(y,z)$ would be the zero polynomial
of degree at most $3$, which is impossible since $L(0)=50$ and the
coefficient of $z^3$ is $100y^3\ne0$. Comparing coefficients of $z^3$
gives $L(h_i)(b+h_i)=100h_i^3$. Consequently,
\[
\omega(u)=\frac{L(u)(b+u)-100u^3}{625}
=u^5+bu^4-\frac45\,u^3-\frac{16}{25}\,bu^2+\frac2{25}\,u+\frac2{25}b.
\]
It remains to put $a=2b/25$.
\end{proof}

\begin{lem}\label{lem:height-moments}
For $0\le a\le a'$, the polynomial $\omega_a$ has five roots
$h_1,\ldots,h_5\in(-1,1)$, counted with multiplicity. For $a<a'$ the
roots are simple, while at $a=a'$ one root is double. The weights
in~\eqref{eq:4-weight} satisfy
\begin{equation}\label{eq:height-moments}
2\sum_{j=1}^5\lambda_jh_j^k=\frac12\int_{-1}^1u^k\,du,\quad k=0,1,2,3,4.
\end{equation}
\end{lem}

\begin{proof}
For $a=0$, the roots are
\[
0,\quad \pm\sqrt{\frac{2+\sqrt2}{5}},\quad \pm\sqrt{\frac{2-\sqrt2}{5}}.
\]
The discriminant of $\omega_a$ is
\[
\frac8{5^{10}}\,(1953125a^4-34000a^2+8) (3828125a^4-4000a^2+128).
\]
The first positive zero of the discriminant is $a'$, while the second
factor is positive on $\mathbb R$, since its discriminant as a quadratic
polynomial in $a^2$ is negative. Moreover,
\[
\omega_a(1)=7/25+11a/2>0,\quad \omega_a(-1)=-7/25+11a/2<0,\quad a\in [0,a'].
\]
Thus the roots remain real and cannot leave $(-1,1)$. At $a=a'$,
the double root is $-\sqrt{4-\sqrt6}/5$. The fact that the
discriminant has a simple zero rules out any other multiple roots.

To verify the moment conditions, put
\[
Q(h)=\lambda(h)^{-1},\quad
J_a(h)=h^4+\frac{25}{2}ah^3-\frac7{15}h^2-\frac{23}{6}ah+\frac1{75}.
\]
A direct calculation gives
\[
Q(h)J_a(h)-2\omega_a'(h)=(30h^3-14h)\omega_a(h).
\]
For a simple root $h_j$, we also have
\[
J_a(h_j)=\frac12\int_{-1}^1\frac{\omega_a(u)}{u-h_j}\,du.
\]
Hence the interpolatory weight at $h_j$ equals
$J_a(h_j)/\omega_a'(h_j)=2\lambda_j$. This proves~\eqref{eq:height-moments}
for $a<a'$, and the case $a=a'$ follows by continuity. The weights are
positive because $Q(h)=30(h^2-2/5)^2+36/5>0$.
\end{proof}

\begin{lem}\label{lem:phase-completion}
For the heights and weights from Lemma~\ref{lem:height-moments},
there exist $\eta_1,\ldots,\eta_5\in\mathbb S^1$ such that, with
$\rho_j=\sqrt{1-h_j^2}$,
\begin{equation}\label{eq:phase4}
\sum_{j=1}^5\lambda_j\rho_j^2h_j^k\eta_j=0,\quad k=0,1,2,\quad
\sum_{j=1}^5\lambda_j\rho_j^4\eta_j^2=0.
\end{equation}
All such choices are obtained from any one of them by a common rotation
and, possibly, complex conjugation.
\end{lem}

\begin{proof}
Form the matrix $M\in\mathbb R^{5\times3}$ with rows
\[
M_{j,:}=\sqrt{2\lambda_j} \Bigl(1,\sqrt3h_j,\frac{\sqrt5}{2}\,(3h_j^2-1)\Bigr).
\]
By~\eqref{eq:height-moments}, $M^TM=I_3$. Choose an orthonormal
complement $Z\in\mathbb R^{5\times2}$:
\[
M^TZ=0,\quad Z^TZ=I_2,\quad ZZ^T=I_5-MM^T.
\]
The weight formula gives
\[
Z_{j1}^2+Z_{j2}^2 =1-2\lambda_j\Bigl(1+3h_j^2+\frac54\,(3h_j^2-1)^2\Bigr)
=\frac{15}{2}\lambda_j\rho_j^4.
\]
We may therefore set
\[
\eta_j=\frac{Z_{j1}+iZ_{j2}}{\sqrt{15\lambda_j/2}\,\rho_j^2}.
\]
The first three equalities in~\eqref{eq:phase4} follow from
$M^TZ=0$, and the last one from $Z^TZ=I_2$.
Conversely, any collection $\eta_j$ satisfying~\eqref{eq:phase4}
defines, by this formula, an orthonormal complement $Z$; here we use
$(15/2)\sum_j\lambda_j\rho_j^4=2$, which follows from
\eqref{eq:height-moments}. Two such complements differ by a
transformation $Z\mapsto ZU$, $U\in O(2)$, giving the stated uniqueness.
\end{proof}

\begin{proof}[Proof of Corollary~\ref{cor:one-parameter-4}]
Choose $\eta_j=e^{2i\varphi_j}$ by Lemma~\ref{lem:phase-completion}
and define a formula by the weighted orbits~\eqref{eq:4-family-nodes}.
Odd azimuthal frequencies vanish because of the half-turn.
Frequency $0$ is integrated exactly by~\eqref{eq:height-moments},
while frequencies $\pm2,\,\pm4$ are handled by~\eqref{eq:phase4}
and its complex conjugates. Thus the formula is exact on
$\Pi_4(\mathbb S^2)$.

All nodes are distinct, including when two heights coincide at the
endpoint of the interval. Indeed, coincidence of two orbits would require
$h_i=h_j$ and $\eta_i=\eta_j$. Then two rows of the orthogonal matrix
$(M\ Z)$ would coincide, which is impossible. The roots can be numbered
continuously in $a$, and analytically for $a<a'$. The orthonormal
complement $Z$ can also be chosen continuously on the interval and
locally analytically for $a<a'$. A continuous choice of arguments of
$\eta_j$ gives the required family. Minimality follows from
$N_{\mathrm{wt}}(4)=10$ (see~\cite{BondarenkoGorbachev}).

For $a=0$, Popov's formula is invariant under the half-turn about the
axis $(\cos(\pi/8),\sin(\pi/8),0)$. Its five pairs have projections
onto this axis equal to the roots of $\omega_0$, with weights given
by~\eqref{eq:4-weight}. Uniqueness in Lemma~\ref{lem:phase-completion}
shows that it agrees with the constructed formula up to a rotation
and possibly a reflection. The latter can be removed using the
reflection symmetry of Popov's formula.

For the other endpoint, put $\delta=\sqrt{4-\sqrt6}$. We have
$\omega_{a'}(1/\delta)=0$ and $\lambda(1/\delta)=1/9$.
The corresponding two nodes $p,q$ satisfy
\[
p\cdot q=\frac2{\delta^2}-1=\frac{\sqrt6-1}{5}.
\]
By Lemma~\ref{lem:projectors}, for $t=4$ the matrix $C$ is positive
semidefinite, with $C_{ii}=\lambda_i^{-1}-9$ and
$C_{ij}=-K_2(x_i\cdot x_j)$ for $i\ne j$.
If the weight of a node is $1/9$, its corresponding diagonal entry
vanishes, and hence the entire row of $C$ vanishes. Therefore, for
any node $x\ne p,q$,
\[
K_2(p\cdot x)=K_2(q\cdot x)=0.
\]
Denote the roots of $K_2$ by $u_\pm=(-1\pm\sqrt6)/5$ and rotate
$p,q$ to $e_3$, $x(u_+,0)$. The other eight nodes are the intersections
of the sphere with pairs of planes $p\cdot x=u_\pm$,
$q\cdot x=u_\pm$. We obtain the following explicit representation
of the endpoint formula:
\[
\begin{array}{c|c|c}
\text{height}&\text{azimuths}&\text{weight of each node}\\ \hline
1&0&1/9\\
u_+&0&1/9\\
u_+&\pm\theta_+&(8+3\sqrt6)/144\\
u_+&\pm\theta_-&(16-\sqrt6)/144\\
u_-&\pi/4+k\pi/2,\ k=0,1,2,3&(16-\sqrt6)/144
\end{array}
\]
where $\theta_\pm\in(0,\pi)$ and
\[
\cos\theta_+=-1+\frac{\sqrt6}{2},\quad
\cos\theta_-=-\frac{6+\sqrt6}{10}.
\]
The weights are obtained from the moments on each latitude.
They are uniquely determined: the factors $(h-u_-)(h-1)$ and
$(h-u_+)(h-1)$ isolate the latitudes, on which interpolation by
trigonometric polynomials of degree $2$ suffices.
This is the Heo--Xu formula~\cite{HeoXu}, with one pole, a square,
and a pentagon. It is symmetric under reflection in the plane
$\xi_2=0$, so here too equivalence under $O(3)$ reduces to
equivalence under rotations.
\end{proof}

Near $a=0$, after fixing rotations, this curve lies in the local
family of Theorem~\ref{thm:main-families}. The entire interval is
constructed separately: the existence of the azimuths follows from
Lemma~\ref{lem:phase-completion}, not merely from the local theorem.
\section{Minimal number of nodes in degree six}\label{app:min6}

We show that every positive cubature formula of degree $6$ has at
least $18$ nodes. Together with Lemma~\ref{lem:popov-exactness},
this gives $N_{\mathrm{wt}}(6)=18$.

Let $(X,\Lambda)$ be a positive formula of degree $6$ with $N$ nodes.
For a polynomial $p$ in one variable, put
\[
\mathcal E(p)=\sum_{i,j=1}^N\lambda_i\lambda_jp(x_i\cdot x_j).
\]
Let $P_l$ denote the Legendre polynomial with $P_l(1)=1$, and put
\[
Z_l(u)=(2l+1)P_l(u),\quad Z_{l,y}(x)=Z_l(x\cdot y).
\]
Since $Z_{l,y}$ is the reproducing kernel for spherical harmonics of
degree $l$, we have
\[
\mathcal E(P_l)=\frac1{2l+1}
\biggl\|\sum_{j=1}^N\lambda_jZ_{l,x_j}\biggr\|_{L^2(\mathbb S^2,\sigma)}^2
\ge0.
\]
Exactness of the formula implies
\begin{equation}\label{eq:kernel-moments6}
\mathcal E(P_0)=1,\quad \mathcal E(P_l)=0,\quad l=1,\ldots,6.
\end{equation}

We also need an upper bound. Let $g=\sum_lg_lP_l$, where $g_l\ge0$.
Then the matrix
\[
M=\bigl(\sqrt{\lambda_i\lambda_j}\,g(x_i\cdot x_j)\bigr)_{i,j=1}^N
\]
is positive semidefinite. By Lemma~\ref{lem:projectors}, the matrix
\[
A=\bigl(\sqrt{\lambda_i\lambda_j}\,K_3(x_i\cdot x_j)\bigr)_{i,j=1}^N
\]
is an orthogonal projector. Hence
\begin{equation}\label{eq:kernel-upper6}
\mathcal E(K_3g)=\operatorname{tr}(AM)
\le\operatorname{tr}M=g(1).
\end{equation}

Choose two polynomials
\[
f(u)=(u+3)\bigl(50u^5+16u^4-33u^3-6u^2+4u+1\bigr)^2,
\]
\[
g(u)=\frac{133}{2}\,P_6(u)+\frac{63}{2}P_8(u).
\]
We have $f\ge0$ on $[-1,1]$, $f(1)=4096$, and $g(1)=98$.
For
\[
r(u)=f(u)-K_3(u)g(u),\quad
K_3(u)=\frac{35u^3+15u^2-15u-3}{2},
\]
the expansion $r=\sum_{l=0}^{11}r_lP_l$ gives
\[
\begin{array}{c|r}
l&r_l\\ \hline
0&7858/55\\
7&-150863/277134\\
8&-32489/81510\\
9&-93928559/510510\\
10&-109375/7106\\
11&-12909175/176358
\end{array}
\]
The coefficients $r_1,\ldots,r_6$ are not needed because
of~\eqref{eq:kernel-moments6}. Since $r_l<0$ for $l=7,\ldots,11$,
we obtain
\[
\mathcal E(r)\le\frac{7858}{55}.
\]
Consequently, by~\eqref{eq:kernel-upper6},
\[
\mathcal E(f)=\mathcal E(r)+\mathcal E(K_3g)
\le\frac{7858}{55}+98=\frac{13248}{55}.
\]
On the other hand, positivity of the weights, nonnegativity of $f$,
and the Cauchy--Schwarz inequality give
\[
\mathcal E(f)\ge f(1)\sum_{j=1}^N\lambda_j^2
\ge\frac{4096}{N}.
\]
Thus
\[
N\ge\frac{4096\cdot55}{13248}=\frac{3520}{207}
=17+\frac1{207}>17.
\]
Hence $N\ge18$. Popov's formula from Lemma~\ref{lem:popov-exactness}
shows that this bound is attained.

\subsection*{Use of artificial intelligence}
Generative artificial intelligence tools were used in preparing this paper.
The problems and preliminary results are the author's own. The author
carried out the final verification and takes responsibility for the
content of the paper.

\end{document}